\documentclass[twoside,leqno]{article}

\usepackage[letterpaper]{geometry}

\usepackage{siamproceedings}

\usepackage[T1]{fontenc}
\usepackage{amsfonts}
\usepackage{graphicx}
\usepackage{epstopdf}
\usepackage{enumitem}
\usepackage{algorithmic}
\ifpdf
  \DeclareGraphicsExtensions{.eps,.pdf,.png,.jpg}
\else
  \DeclareGraphicsExtensions{.eps}
\fi

\newsiamremark{remark}{Remark}
\newsiamremark{hypothesis}{Hypothesis}
\crefname{hypothesis}{Hypothesis}{Hypotheses}
\newsiamthm{claim}{Claim}

\usepackage{amsopn}
\usepackage{xurl}

\usepackage{tikz}
\usepackage{amssymb}
\usepackage{mathtools}

\newcommand{\EE}{\mathbb{E}}  
\newcommand{\PP}{\mathbb{P}}  
\newcommand{\RR}{\mathbb{R}}  %
\newcommand{\NNN}{\mathcal{N}}  
\newcommand{\NN}{\mathbb{N}}  
\newcommand{\AAA}{\mathcal{A}} 
\newcommand{\1}{\mathbf{1}} 
\newcommand{\eps}{\varepsilon}

\DeclareMathOperator{\cov}{Cov}
\DeclareMathOperator{\polylog}{polylog}
\DeclareMathOperator{\tr}{tr}
\DeclareMathOperator{\Tr}{Tr}
\DeclareMathOperator{\vecf}{vec}
\newcommand{\defeq}{\coloneqq}

\newcommand{\Ygauss}{Y_{\mathrm{Gauss}}}
\newcommand{\Xfree}{X_{\mathrm{free}}}
\newcommand{\NC}{\mathrm{NC}}
\newcommand{\maxV}{\vee}
\newcommand{\minV}{\wedge}

\newcommand{\flatta}[2]{\AAA_{[\,#1\,\mid\,#2\,]}}

\newsiamthm{conjecture}{Conjecture}

\begin{document}

\newcommand\relatedversion{}

\title{\Large Random Matrices, Intrinsic Freeness, and Sharp Non-Asymptotic Inequalities\relatedversion}
    \author{Afonso S. Bandeira\thanks{ETH Z\"urich, Switzerland (\email{bandeira@math.ethz.ch}, \url{https://people.math.ethz.ch/\~ abandeira/}).}
}

\date{}

\maketitle

\begin{abstract} 
Random matrix theory has played a major role in several areas of pure and applied mathematics, as well as statistics, physics, and computer science. This lecture aims to describe the intrinsic freeness phenomenon and how it provides new easy-to-use sharp non-asymptotic bounds on the spectrum of general random matrices.
We will also present a couple of illustrative applications in high dimensional statistical inference.

This article accompanies a lecture that will be given by the author at the International Congress of Mathematicians in Philadelphia in the Summer of 2026. 
\end{abstract}

\section{Introduction.}

Random matrix theory has played a major role in several areas of pure and applied mathematics, as well as physics and computer science. Furthermore, these different interactions have motivated different lines of inquiry. While a complete historical account of the study of random matrices by various mathematical communities is beyond the scope of this short text, we start with an abridged description, before reaching the main objects of study of this survey.  

The connections between random matrices and statistics date back at least to the 1920s, when 
Wishart~\cite{Wishart1928} studied the distribution of sample covariance matrices of samples from a Gaussian distribution. In 1967 Mar\v{c}enko and Pastur~\cite{marchenko1967} worked out the eigenvalue distribution of these random matrices, now called Wishart matrices, this distribution is now known as the Mar\v{c}enko-Pastur distribution.

In the 1950s 
Wigner was interested in studying eigenvalues of certain matrices arising in nuclear physics and realized that one could instead study eigenvalues of large random Hermitian matrices, now called Wigner matrices~\cite{wigner1951}. In 1958, Wigner~\cite{wigner1958} showed that the spectral distribution of Wigner matrices converges to the celebrated semi-circular law. These discoveries have since motivated countless fascinating mathematical inquiries about the spectrum of many classes of random matrices, and have made random matrix theory a core part of mathematical physics.

The connections with computation also have a rich history. 
In a pair of seminal papers in 1947 and 1951, 
von Neumann and 
Goldstine~\cite{vonneumann1947,vonneumann1951}, were interested in studying potential cancellation effects in the accumulation of errors in numerical algorithms for solving linear systems on ``typical'' matrices. They refer to work of 
Bargmann that studies the condition number of certain matrices with random entries. In 1988 
Edelman~\cite{edelman1988} computed the asymptotics of the condition number of a matrix with independent Gaussian entries.~\footnote{In his PhD thesis, Edelman mentions that the paper of Bargmann that is referenced in footnote 24 of~\cite{vonneumann1947} is unlikely to be available.} 

Random matrices have beautiful connections with many other areas of Mathematics. Highlights include the connections with Analytic Number Theory and the Hilbert-Polya conjecture regarding the zeros of the Riemann zeta function (see~\cite{rudnick-sarnak1996,katz-sarnak1999}), and the work of Haagerup and Thorbj{\o}rnsen~\cite{HaagerupThorbjornsen2005} (which we will mention again below) that uses random matrix theory to solve a long standing question in operator algebras.

Random matrices have since been studied from several different perspectives.\footnote{There are several excellent monographs on Random Matrices, some of the author's favorite are~\cite{tao2012,andersonguionnetzeitouni-book-2010,Tro15:Introduction-Matrix}.} The most classical, following the line of work started by Wigner, is the asymptotic study of specific random matrix ensembles with strong symmetries, such Wigner matrices which correspond to self-adjoint random matrices with iid entries above the diagonal (and iid entries in the diagonal, potentially with a different distribution). This line of work has produced incredibly precise estimates on spectral properties of these classes of random matrices including local laws of eigenvalues and delocalization of eigenvectors under mild conditions on the entrywise distribution (see, e.g.,~\cite{erdosetal2010,taovu2010}).

\begin{definition}[Standard Wigner Matrix]\label{def:standardWignerMatrix}
We will call a $d\times d$ symmetric random matrix a standard Wigner matrix when the upper triangular entries are i.i.d. standard Gaussian, i.e. for all $1\leq i\leq j\leq d$ we have $W_{ij}\sim \NNN\big(0,\frac1d\big)$, all independent.\footnote{When the diagonal entries are distributed accordingly to $\NNN\big(0,\frac2d\big)$ and the off-diagonal ones as $\NNN\big(0,\frac1d\big)$ the matrix is invariant to orthogonal conjugation and this case is refereed to as the \emph{Gaussian Orthogonal Ensemble (GOE)}. We note that for a standard Wigner matrix, $\EE W^2 = I$.}
\end{definition}

In applications in applied mathematics, statistics and computer science one is often mostly concerned with extremal eigenvalues (or singular values) but requires non-asymptotic bounds that can be used in large, but fixed, dimension. A notable line of work (see this ICM 2010 publication by Rudelson and Vershynin~\cite{RudelsonVershynin2010}) develops non-asymptotic bounds for extremal singular values of random matrices with iid entries. These are often proved by writing extremal singular (or eigen) values as empirical processes (via the Courant-Fisher variational formula) and using geometric tools involving covering numbers. In many applications, however, the random matrices that need to be analyzed do not have independent entries and a different approach appears to be required.\footnote{It is an open question to give a proof of the Non-commutative Khintchine inequality (Theorem~\ref{thm:NCK-Ak}), even up to polylogarithmic factors, using a geometric approach, see~\cite{bandeira2025tensor}.}

A perspective able to handle matrices with dependent entries comes from operator space theory, the study of non-commutative Banach spaces~\cite{Pisier2003IntroductionTO}. It was in this context that the celebrated non-communitative Khintchine inequality of Lust-Piquard and Pisier~\cite{lustpiquard1986,lustpiquardpisier1991} was shown (see~\cite[\S 9.8]{Pisier2003IntroductionTO}).

\begin{theorem}[Non-commutative Khintchine inequality~\cite{lustpiquard1986,lustpiquardpisier1991,Pisier2003IntroductionTO}]\label{thm:NCK-Ak}
Let $A_1,\dots, A_n \in \RR^{d_1\times d_2}$ and $g_1,\dots,g_n$ iid $\NNN(0,1)$
\begin{equation}\label{eq:nck1}
\sigma \lesssim \EE\left\| \sum_{k=1}^n g_k A_k \right\| \lesssim \sqrt{\log(d_1+d_2)} \ \sigma, 
\end{equation}
where $\sigma^2 = \left\| \sum_{k=1}^n A_kA_k^\top \right\| \maxV \left\| \sum_{k=1}^n A_k^\top A_k \right\|$, and $a\maxV b$ denotes the maximum of $a$ and $b$.
\end{theorem}

Another important line of work studies sums of independent random matrices. This approach, referred to as \emph{Matrix Concentration}~\cite{Tro15:Introduction-Matrix}, dates back to work in Quantum Information Theory by Ahlswede and Winter in the early 2000s~\cite{AhlswedeWinter2002}. The idea is to use a matrix version of the ``Chernoff-trick'' and bound a matrix version of the moment generating function of a random matrix using operator inequalities, such as Golden-Thompson or Lieb's concavity. Inequalities such as the Matrix Bernstein Inequality of Oliveira and Tropp~\cite{Oliveira1-2010,Oliveira2-2010,Tro12:User-Friendly-Tail,Tro15:Introduction-Matrix} have found an incredible amount of applications and have become part of the standard toolbox of high dimensional probability theory.

\begin{theorem}[Matrix Bernstein~\cite{Oliveira1-2010,Oliveira2-2010,Tro12:User-Friendly-Tail,Tro15:Introduction-Matrix}]\label{thm:matrixBernstein}
 Let $\{H_k\}_{k=1}^n$ be a sequence of independent random symmetric $d\times d$ matrices. Assume that each $H_k$ satisfies:
 \[
  \mathbb{E} H_k = 0 \text{ and } \lambda_{\max}\left( H_k \right) \leq R \text{ almost surely.}
 \]
Then, for all $t\geq 0$,
\begin{equation}\label{eq:matrixbernstein}
 \PP\left[ \lambda_{\max}\left( \sum_{k=1}^n H_k \right) \geq t \right] \leq d \cdot \exp\left( \frac{-t^2}{2\sigma^2 + \frac23Rt} \right) \text{ where } \sigma^2 = \left\| \sum_{k=1}^n\mathbb{E}\left(H_k^2\right) \right\|.
\end{equation}
\end{theorem}

These two inequalities (Theorems~\ref{thm:NCK-Ak} and~\ref{thm:matrixBernstein}) are tightly connected. While~\eqref{eq:matrixbernstein} is a tail bound it is often best to use it as a bound on $\EE\|\sum H_k\|$ followed by scalar concentration inequalities~\cite[\S 1.6.5]{Tro15:Introduction-Matrix}. In fact, a standard application of Gaussian concentration provides a tail bound version of~\eqref{eq:nck1} with the deviations controlled by a weak variance parameter $\sigma_\ast(\AAA)$ that smaller (or equal) to $\sigma(\AAA)$ (see~\cite[\S\S 8, 9]{MDS-Book-2025} for a pedagogical treatment).
Furthermore, using symmetrization, Tropp~\cite{Tropp_MatrixConcentrationElementary} gave a proof for a $\EE\|\sum H_k\|$ version of~\eqref{eq:matrixbernstein} using~\eqref{eq:nck1}.\footnote{For $H_1,\dots,H_n$ $d\times d$ self-adjoint centered independent random matrices, symmetrization gives: $\EE\left\| \sum H_k \right\| \lesssim \EE_{H}\EE_{g} \left\| \sum g_k H_k \right\| \lesssim \sqrt{\log d}\, \EE_{H}  \left\| \sum_{k=1}^n H_k^2 \right\|^{\frac12}$, which can then be estimated (see~\cite{Tropp_MatrixConcentrationElementary,MDS-Book-2025}).} 

\begin{remark}[Hermitian dilation]\label{rem:HermitianDilation}
We note that even though we formulate some of the inequalities for self-adjoint matrices this is not restrictive as, given a non symmetric matrix $X\in\RR^{d_1\times d_2}$ one can use the inequalities on the Hermitian dilation $
\bigl[ \begin{smallmatrix} 0 & X \\ 
  X^\top & 0 \end{smallmatrix} \bigr]$. Note that $\|\bigl[ \begin{smallmatrix} 0 & X \\ 
  X^\top & 0 \end{smallmatrix} \bigr]\|=\|X\|$ and $\|\EE \bigl[ \begin{smallmatrix} 0 & X \\ 
  X^\top & 0 \end{smallmatrix} \bigr]^2\| = \|\EE X X^\top\| \maxV \|\EE X^\top X\|$. 
\end{remark}

The line of work that this article aims to describe starts with the observation that the dimensional factors in Theorems~\ref{thm:NCK-Ak} and~\ref{thm:matrixBernstein} are suboptimal in important examples. The reader is invited to try Theorem~\ref{thm:NCK-Ak} with two instructive examples, (i) a diagonal random matrix whose diagonal entries are iid $\NNN(0,1)$ and (ii) a Gaussian Wigner matrix: the first example shows that the dimensional factor cannot be removed always, while the second shows that the inequality is not sharp even in the classical example of Wigner matrices. The particular case of independent entries (which include these two examples) was relatively well understood in the mid 2010s~\cite{bandeira2016sharp,latala2018dimension}.

\subsection*{Aims and audience:} The main goal of this survey is to inspire graduate students in Mathematics and related areas to (i) work on improving our understanding of random matrices and/or (ii) use matrix concentration inequalities in their fields. With this in mind, it does not try to be exhaustive or heavy on details. It aims to showcase the most important ideas, while being light on the required background, and giving pointers for the interested reader to find more. The author hopes the reader enjoys reading it as much as he enjoyed writing it.

\subsection*{Outline}

After the introduction above, Section~\ref{sec:intrinsicfreeness} is the core of this survey: it explains the intrinsic freeness phenomenon, with asymptotic freeness addressed in \S\ref{sec:asymptoticfreeness} and non-asymptotic inequalities in \S\ref{sec:nonasymptoticfreeness}; universality is discussed in \S\ref{sec:universality}. Section~\ref{sec:extensionsapplications} is a brief description of some extensions and applications, deferring more comprehensive accounts of each topic to further references: \S\ref{sec:sharptransitions} and \S\ref{sec:matrixchaos} respectively discuss sharp phase transitions in random matrix models, and generalizations of matrix inequalities to matrix chaoses; \S\ref{sec:tensorPCA} briefly introduces the tensor PCA problem in the interface of high dimensional statistics and theoretical computer science and \S\ref{sec:kikuchi} and \S\ref{sec:soschaos} show two applications of the tools in this survey to this problem; lastly, an application of the intrinsic freeness phenomenon to the matrix Spencer conjecture is highlighted in \S\ref{sec:matrixspencer}. 

\subsection*{Notation} We make several notational choices: For $X$ a $d\times d$ matrix, $\|X\|$ denotes the spectral norm, $\Tr(X)$ denotes the trace $\Tr(X)=\sum_{i=1}^d X_{ii}$ and $\tr(X)$ the normalized trace $\tr(X)=\frac1d\Tr(X)$. $\mathbb{S}^{d-1}$ denotes the unit sphere in $\mathbb{R}^d$. In expressions such as $\tr X^p$ or $\EE X^p$ the power binds before the trace or expectation. $a\sim_q b$ means that there exists a constant $C_q>0$, potentially depending on $q$, such that $a/C_q\leq b\leq C_q$. We also us standard big-O notation, where $a_n=\Omega(b_n)$ means that $\limsup b_n/a_n\leq C$ for a constant $C>0$, and $a_n=o(b_n)$ that $\limsup a_n/b_n=0$, for positive sequences $a_n,b_n$. $\tilde{\Omega}$ indicates a possibility for hidden polylogarithmic factors. 



\section{Intrinsic Freeness}\label{sec:intrinsicfreeness}

Our main object of study in this section will be a $d \times d$ self-adjoint centered random matrix $X$ whose entries are jointly Gaussian.\footnote{While in our exposition we chose to treat matrices with real entries, the theory is essentially unchanged for complex valued matrices (by replacing $^\top$ by $^\ast$, and ``symmetric'' by ``Hermitian'').} Such a random matrix $X$ can always be written as $X = \sum_{k=1}^n g_k A_k$ for $A_1,\dots,A_k$ deterministic $d\times d$ symmetric matrices and $g_1,\dots,g_n$ iid $\NNN(0,1)$. Note that $\EE X^2 = \sum_{k=1}^n A_k^2$.

In order to show the source of the dimensional factor in matrix concentration inequalities, we will start with a proof of the Non-commutative Khintchine inequality (Theorem~\ref{thm:NCK-Ak}). The argument involves computing mixed moments of standard gaussians of the form $\EE[g_{u(1)}\cdots g_{u(p)}]$. The main from of cancellation arises from the fact that such moments can only be non-zero if each index appears an even number of times. These calculations are elegantly organized by the notion of pair partitions and Wick's formula (see~\cite[\S 8]{MDS-Book-2025} for a pedagogical treatment in the same notation; and Figure~\ref{fig:WicksFormula}).

\begin{definition}[Pair Partition]\label{def:pairpartition}
Given $k$ a positive integer, we define $\PP_2[k]$ as the set of partitions of $[k]$ into subsets of size $2$ each. If $k$ is odd then $\PP_2[k]$ is empty. Given a function $u$ on $[k]$ and a pair partition $\nu\in\PP_2[k]$ we say that $u$ is compatible with $\nu$, and write $u\sim\nu$ if for all sets $(i,j)\in\nu$ we have $u(i)=u(j)$. Given even $k=2p$ and a partition $\nu\in\PP_2[2p]$ we define the $\nu$-assignment $u_\nu:[2p]\to[p]$ as the surjective function that is compatible with $\nu$ and no other partition. Any of the $p!$ such functions works for our purposes, but we pick the first in lexicographic order.
\end{definition}

\begin{lemma}[Wick's formula]\label{lemma:Wicksformula}
Let $g_1,\dots,g_n$ be iid $\NNN(0,1)$ random variables and let $u:[2p]\to[n]$ then
\begin{equation}\label{eq:Wicksformula}
\EE[g_{u(1)}\cdots g_{u(2p)}] = \sum_{\nu\in \PP_2[2p]} 1_{u \sim \nu},
\end{equation}
where $\PP_2[2p]$ denotes a set of pair partitions, and $u \sim \nu$ means that the function $u$ is compatible with $\nu$. 
\end{lemma}

\begin{figure}[h]
\centering
\begin{tikzpicture}
\node[left] (0,0) {$\nu_1 = ~$};

\draw[thick] (0,0) -- (3.5,0);
\foreach \i in {1,...,8}
{
         \draw[fill=black] (0.5*\i-0.5,0) circle (0.05);
}
{
\fill[gray!30, opacity=.6, rounded corners=5pt]
       (1.30,-0.2) rectangle (1.70,0.2);
\fill[gray!30, opacity=.6, rounded corners=5pt]
       (-0.20,-0.2) rectangle (0.20,0.2);
\fill[gray!30, opacity=.6, rounded corners=5pt]
       (2.80,-0.2) rectangle (3.20,0.2);
\fill[gray!30, opacity=.6, rounded corners=5pt]
       (3.30,-0.2) rectangle (3.70,0.2);
}
\begin{scope}
     \clip (-0.1,0) rectangle (3.6,1.1);
     \draw[thick] (1.5,-.68) circle (1.65);
     \draw[thick] (0.75,0) circle (0.25);
     \draw[thick] (2.5,0) circle (1);
     \draw[thick] (2.25,0) circle (0.25);
\end{scope}

\begin{scope}[xshift=5.5cm]
\node[left] (0,0) {$\nu_2 = ~$};

\draw[thick] (0,0) -- (3.5,0);
\foreach \i in {1,...,8}
{
         \draw[fill=black] (0.5*\i-0.5,0) circle (0.05);
}
{
\fill[gray!30, opacity=.6, rounded corners=5pt]
       (1.30,-0.2) rectangle (1.70,0.2);
\fill[gray!30, opacity=.6, rounded corners=5pt]
       (-0.20,-0.2) rectangle (0.20,0.2);
\fill[gray!30, opacity=.6, rounded corners=5pt]
       (2.80,-0.2) rectangle (3.20,0.2);
\fill[gray!30, opacity=.6, rounded corners=5pt]
       (3.30,-0.2) rectangle (3.70,0.2);
}
\begin{scope}
     \clip (-0.1,0) rectangle (3.6,1.4);
     \draw[thick] (1.75,-0.55) circle (1.85);
     \draw[thick] (0.75,0) circle (0.25);
     \draw[thick] (2.25,0) circle (0.75);
     \draw[thick] (2.25,0) circle (0.25);
\end{scope}
\end{scope}

\begin{scope}[xshift=11cm]
\node[left] (0,0) {$\nu_3 = ~$};

\draw[thick] (0,0) -- (3.5,0);
\foreach \i in {1,...,8}
{
         \draw[fill=black] (0.5*\i-0.5,0) circle (0.05);
}
{
\fill[gray!30, opacity=.6, rounded corners=5pt]
       (1.30,-0.2) rectangle (1.70,0.2);
\fill[gray!30, opacity=.6, rounded corners=5pt]
       (-0.20,-0.2) rectangle (0.20,0.2);
\fill[gray!30, opacity=.6, rounded corners=5pt]
       (2.80,-0.2) rectangle (3.20,0.2);
\fill[gray!30, opacity=.6, rounded corners=5pt]
       (3.30,-0.2) rectangle (3.70,0.2);
}

\begin{scope}
     \clip (-0.1,0) rectangle (3.6,1.4);
     \draw[thick] (0.75,0) circle (0.75);
     \draw[thick] (0.75,0) circle (0.25);
     \draw[thick] (3.25,0) circle (0.25);
     \draw[thick] (2.25,0) circle (0.25);
\end{scope}
\end{scope}

\end{tikzpicture}
\caption{Visualization of the three pairings in $\PP[8]$ compatible with $u:[8]\to[3]$ given by $u(1)=u(4)=u(7)=u(8)=1$, $u(2)=u(3)=2$, and $u(5)=u(6)=3$. The nodes $1,4,7,8$ are shaded. Indeed, $\EE g_1^4g_2^2g_3^2 = \EE g_1^4 = 3$ (by independence and the fact that $\EE g^2=1$ and $\EE g^4=3$ for a standard Gaussian). Wick's formula encodes the fact that $q$-th moment of a standard gaussian is given by the number of perfect matchings of a $K_q$ 
graph.}
\label{fig:WicksFormula}
\end{figure}

\begin{proof}[Proof of upper bound in Theorem~\ref{thm:NCK-Ak}]
Using Hermitian dilation (Remark~\ref{rem:HermitianDilation}) we reduce to the case of $d\times d$ symmetric matrices $X=\sum_{k=1}^ng_kA_k$. Let $p$ be a positive integer. By Jensen's inequality, $\left( \EE \|X\| \right)^{2p} \leq  \EE \|X\|^{2p} = \EE \left\|X^{2p}\right\|$. Since $X^{2p}\succeq 0$, the spectral norm is bounded by the trace $\left\|X^{2p}\right\|\leq d\tr \left( X^{2p} \right)$ (recall that $\tr(X) = \frac1d\Tr(X)$ denotes de normalized trace). We now focus on bounding $\EE \tr  X^{2p}.$ Using Wick's formula:

\begin{eqnarray}
\EE \tr X^{2p} &=& \sum_{u:[2p]\to[n]} \EE[g_{u(1)}\cdots g_{u(2p)}] \tr\left(A_{u(1)}\cdots A_{u(2p)}\right) \nonumber \\
&=& \sum_{u:[2p]\to[n]} \sum_{\nu\in \PP_2[2p]} 1_{u \sim \nu}\tr\left(A_{u(1)}\cdots A_{u(2p)}\right) 
= \sum_{\nu\in \PP_2[2p]}\sum_{\substack{u:[2p]\to[n] \\ u\sim\nu} } \tr\left(A_{u(1)}\cdots A_{u(2p)}\right).\label{eq:NCKproofEq1}
\end{eqnarray}
If the matrices $A_k$ were commutative then the summands 
$$\sum_{\substack{u:[2p]\to[n] \\ u\sim\nu} } \tr\left(A_{u(1)}\cdots A_{u(2p)}\right)$$ would coincide for all pair partitions $\nu$. The summand corresponding to $\nu_0 \defeq \{(1,2),(3,4),\dots,(2p-1,2p)\}$ is particularly elegant
\begin{equation}\label{eq:wickSummandNu0}
\sum_{\substack{u:[2p]\to[n] \\ u\sim\nu_0} } \tr\left(A_{u(1)}\cdots A_{u(2p)}\right) = \sum_{\substack{u:[p]\to[n] } } \tr\left(A_{u(1)}^2\cdots A_{u(p)}^2\right)=\tr\left( \sum_{k=1}^n A_k^2\right)^p,
\end{equation}
recall that the power binds before the trace.
An argument of Buchholz (Lemma~\ref{lemma:commutativeisworstWicks} below) shows a ``commutative is the worst-case'' inequality, in the sense that every summand is bounded by \eqref{eq:wickSummandNu0}. Together with \eqref{eq:NCKproofEq1} it gives:
\begin{equation}\label{eq:proofNCK-P2}
\EE \tr X^{2p} \leq \Big| \PP_2[2p] \Big| \tr\left( \sum_{k=1}^n A_k^2\right)^p.
\end{equation}
A simple combinatorial argument shows $\Big| \PP_2[2p] \Big| = (2p-1)!! \leq (2p)^p$. Since the normalized trace is bounded by the spectral norm, we have:
\[
\EE \|X\| \leq d^{\frac1{2p}}\left( \EE \tr X^{2p} \right)^{\frac1{2p}} \leq d^{\frac1{2p}}\left( (2p)^p \left\| \sum_{k=1}^n A_k^2\right\|^p \right)^{\frac1{2p}} = d^{\frac1{2p}}\sqrt{2p}\, \sigma(X),
\]
where $\sigma(X)^2 = \left\|\EE X^2\right\| = \left\|\sum_{k=1}^n A_k^2\right\|$. Taking $p=\lceil \log d \rceil$ finishes the argument.
\end{proof}

\begin{lemma}\label{lemma:commutativeisworstWicks}[Commutative is the worst-case~\cite{Buchholz01}]
For any $\nu\in\PP_2[2p]$ and $A_1\dots,A_{n}$ symmetric matrices
\[
\sum_{\substack{u:[2p]\to[n] \\ u\sim\nu} } \tr\left(A_{u(1)}\cdots A_{u(2p)}\right) \leq \tr\left( \sum_{k=1}^n A_k^2\right)^p.
\]
\end{lemma}

It is instructive to consider what we will refer to as the \emph{isotropic} case, when \( \EE X^2 = \sum_{k=1}^n A_k^2 = \sigma(X) I \) is a multiple of the identity.\footnote{Note that $\EE W^2=I$, for $W$ a standard Wigner matrix.} In that case, there are many pair partitions that match \eqref{eq:wickSummandNu0}: whenever there is an adjacent pair, it can be ``peeled-off'' in the sum and potentially make adjacent pairs that were not adjacent before; for example, if $\sum_{k=1}^n A_k^2 = \sigma(X)^2 I$ then, for $p=2$ and $\nu=\{(1,4),(2,3)\}$ we have
\[
\sum_{\substack{u:[4]\to[n] \\ u\sim\nu} } \tr\left(A_{u(1)}A_{u(2)}A_{u(3)}A_{u(4)}\right) = \sum_{u:[2]\to[n]  } \tr\left(A_{u(1)}A_{u(2)}A_{u(2)}A_{u(1)}\right) =
\sum_{u:[1]\to[n]  } \tr\left(A_{u(1)}\sigma(X)^2IA_{u(1)}\right) =
\sigma(X)^4.
\]
the pair partitions that can be fully ``peeled-off'' this way are precisely the so-called \emph{non-crossing} partitions (see Figure~\ref{fig:2pairings}).  

\begin{definition}[Crossing and Non-crossing Partitions]\label{partitions:crossing}
We say $\nu\in\PP_2[2p]$ is a crossing partition when it has pairs $(i_1,i_2)\in\nu$ and $(j_1,j_2)\in\nu$ such that $i_1< j_1 < i_2 < j_2$. Otherwise we say $\nu$ is \emph{non-crossing}. The set of non-crossing partition is denoted by $\NC_2[2p]\subset \PP_2[2p]$.
\end{definition}

The argument above suggests that if one is to improve Theorem~\eqref{thm:NCK-Ak} with extra cancellations, they ought to arise from crossings.
Indeed, if all summands corresponding to crossing partitions were suppressed, the super-exponential $\Big| \PP_2[2p] \Big| = (2p-1)!!$ factor in~\ref{eq:proofNCK-P2} would be replaced with $\Big| \NC_2[2p] \Big| \leq 4^p$ which would ultimately lead to a bound without a logarithmic factor (see the ``proof idea'' for Theorem~\ref{prop:Xfreesigma2sigma} below). 

In the second half of the 2010s, Tropp~\cite{Tro18:Second-Order-Matrix} had the key idea of quantify these cancellations
\footnote{The approach in~\cite{Tro18:Second-Order-Matrix} uses Gaussian integration by parts to build a recurrence to bound $\EE\tr X^{2p}$, where crossings also arise and are tightly connected to crossings in Wick's formula.}
with the following \emph{matrix alignment} parameter:
$$
	w(X) \defeq 
	\sup_{U,V,W\in U(d)} \|\mathbf{E}[X_1 U X_2 V X_1 W X_2]\|^{\frac{1}{4}}
	=
	\sup_{U,V,W\in U(d)}\Bigg\|
	\sum_{i,j=1}^n A_i U A_j V A_i W A_j\Bigg\|^{\frac{1}{4}},
$$
where $X_1,X_2$ are i.i.d.\ copies of $X$ and the supremum is taken over 
all (nonrandom) unitary $d \times d$ matrices $U,V,W$.
When all $A_i$ commute, $w(X)\ge 
\|\sum_{ij}A_iA_jA_iA_j\|^{\frac{1}{4}}= 
\|(\sum_{i}A_i^2)^2\|^{\frac{1}{4}}=\sigma(X)$, but if $w(X)\ll\sigma(X)$, we expect cancellations to arise from crossings. Tropp~\cite{Tro18:Second-Order-Matrix} used this quantity to show an improvement of Theorem~\ref{thm:NCK-Ak}: $\EE\|X\| \lesssim \log(d)^{\frac14}\sigma(X) + \log(d)^{\frac12}w(X)$ capturing cancellations arising from non-commutativity (while unfortunately still having a sometimes spurious dimensional factor). While the matrix alignment parameter $w(X)$ appears too difficult to compute in practice, the idea to use such parameters to control crossing cancellations plays a key role in the sequel.


\subsection{Asymptotic Freeness:}\label{sec:asymptoticfreeness}

Stepping back a few decades, cancellations in non-crossing partitions are at the heart of Free Probability~\cite{Voiculescu1991,NicaSpeicher-FreeProbability}, a theory introduced by Voiculescu in the 1980s to tackle problems in operator algebras. It is a non-commutative analogue of probability theory where the concept of \emph{Freeness} plays the role of a non-commutative version of independence.
The connection between free probability and random matrices dates back to the early 1990s when Voiculescu showed that random matrices drawn from certain distributions are asymptotically free~\cite{Voiculescu1991}. Indeed, for our purposes, we can view free probability as providing an asymptotic description of the behavior of Wigner matrices as their dimension grows. One of the central objects in free probability is the notion of a \emph{free semicircular family} $s_1,\dots,s_n$, together with a trace $\tau$ in the algebra they generate.\footnote{To keep the required background light we will not formally define the objects in free probability ($s_1,\dots,s_n$ are infinite dimensional operators) and just discuss them implicitly; a treatment of the content of this section where the objects are formally defined can be found in~\cite[\S 4.1]{bandeira2023free}, and for an introduction to free probability the author recommends the excellent book of Nica and Speicher~\cite{NicaSpeicher-FreeProbability}. See also Definition~\ref{def:Xfree}.}
A free semicircular family $s_1,\dots,s_n$ can be viewed as the limiting objects associated with $W_1^{(N)},\dots,W_n^{(N)}$, $N\times N$ independent standard Wigner matrices (recall Definition~\ref{def:standardWignerMatrix}) as $N\to\infty$.\footnote{In a certain sense, not unlike how Gaussian random variables, a central object in classical probability, can be viewed as the limiting object of binomial random variables. In fact, the semicircular spectral distribution $\frac1{2\pi}\sqrt{4-x^2}\,1_{|x|\leq 2}$ is the limiting distribution arising in the Free Central Limit Theorem (see, for example~\cite{NicaSpeicher-FreeProbability}.)}

Voiculescu's~\cite{Voiculescu1991} asymptotic freeness can then be written as
\begin{equation}\label{eq:VoiculescuAF}
    \lim_{N\to\infty}\EE\left[\tr P\left(W_1^{(N)},\dots,W_n^{(N)}\right) \right] = \tau\big(P\left(s_1,\dots,s_n\right)\big),
\end{equation}
where $P$ is a non-commutative polynomial. For example, $P(X,Y)=X^2Y^2-XYXY$ would reduce to the zero polynomial in a commutative algebra, but in a non-commutative setting it does not: for $X$ and $Y$ $d\times d$ matrices, $P(X,Y)$ is not the zero polynomial.\footnote{When dealing with non self-adjoint matrices (or operators) it makes sense to consider polynomials on $X,Y,\dots$ and their adjoints  $X^\ast,Y^\ast,\dots$. The natural context for this is a $C^{\ast}$ algebra (an algebra with a notion of a norm, and with an involution $\ast$ corresponding to taking the adjoint, satisfying several compatibility conditions), we will not require this formalism for our exposition and refer the interest reader to~\cite{Pisier2003IntroductionTO} and references therein (we note that even if $X$ and $Y$ are self-adjoint, $XY$ may not be).}

In 2005, Haagerup and Thorbj{\o}rnsen
\cite{HaagerupThorbjornsen2005} showed a significant strengthening of \eqref{eq:VoiculescuAF} by showing  \emph{strong asymptotic freeness} (convergence in norm)
\begin{equation}\label{eq:HT05-SAF}
    \lim_{N\to\infty}\EE\left\| P\left(W_1^{(N)},\dots,W_n^{(N)}\right) \right\| = \big\|P\left(s_1,\dots,s_n\right)\big\|.
\end{equation}

One way one can think of \eqref{eq:VoiculescuAF} is by looking at what it says about traces of mixed moments of large standard Wigner matrices (Proposition~\ref{prop:WicksWigner} can be proved directly\footnote{The calculations involved in proving Proposition~\ref{prop:WicksWigner} are particularly elegant when $W_1,\dots,W_n$ are GUEs, a unitary-invariant complex-valued version of Wigner matrices (see~\cite[\S 9.3.1]{MDS-Book-2025}).}). Recall that $\EE \big(W_k^{(N)}\big)^2=I$.
\begin{proposition}\label{prop:WicksWigner}
    Let $W_1,\dots,W_n$ be $N\times N$ independent standard Wigner matrices (recall Definition~\ref{def:standardWignerMatrix}). Given a partition $\nu\in\PP[2p]$ let $u_\nu$ be the $\nu$-assignment (see Definition~\ref{def:pairpartition}). We have
    \begin{equation}\label{eq:WignerNONcrossingpairings}
            \EE\left[\tr W_{u_\nu(1)}\cdots W_{u_\nu(2p)}\right]  = 1,
    \end{equation}
    if $\nu\in \NC[2p]$ is non-crossing, and 
\begin{equation}\label{eq:Wignercrossingpairings}
            \lim_{N\to\infty} \EE\left[\tr W_{u_\nu(1)}\cdots W_{u_\nu(2p)}\right] = 0,
    \end{equation}
if $\nu\in \PP[2p]\setminus \NC[2p]$ is crossing.

    Furthermore, $W_1,\dots,W_n$ enjoy (in the limit) a Wick's formula summing only over non-crossing partitions: For $u:[2p]\to[n]$ we have
    \begin{equation}\label{eq:WicksforWigner}
\lim_{N\to\infty}\EE\left[\tr W_{u(1)}\cdots W_{u(2p)}\right] = \sum_{\nu\in \NC_2[2p]} 1_{u \sim \nu}.
    \end{equation}
The identities \eqref{eq:WignerNONcrossingpairings}--\eqref{eq:WicksforWigner} hold for a free semicircular family $s_1,\dots,s_n$ (without needing to take a limit).
\end{proposition}

\begin{figure}[h]
\centering
\begin{tikzpicture}
\node[left] (0,0) {$\nu_1 = ~$};

\draw[thick] (0,0) -- (3.5,0);
\foreach \i in {1,...,8}
{
         \draw[fill=black] (0.5*\i-0.5,0) circle (0.05);
}
\begin{scope}
     \clip (-0.1,0) rectangle (3.6,1.4);
     \draw[thick] (1.75,-0.55) circle (1.85);
     \draw[thick] (0.75,0) circle (0.25);
     \draw[thick] (2.25,0) circle (0.75);
     \draw[thick] (2.25,0) circle (0.25);
\end{scope}

\begin{scope}[xshift=8cm]
\node[left] (0,0) {$\nu_2 = ~$};

\draw[thick] (0,0) -- (3.5,0);
\foreach \i in {1,...,8}
{
         \draw[fill=black] (0.5*\i-0.5,0) circle (0.05);
}
\begin{scope}
     \clip (-0.1,0) rectangle (3.6,1.1);
     \draw[thick] (1.5,-.68) circle (1.65);
     \draw[thick] (0.75,0) circle (0.25);
     \draw[thick] (2.5,0) circle (1);
     \draw[thick] (2.25,0) circle (0.25);
\end{scope}
\end{scope}
\end{tikzpicture}
\caption{Two examples of pairing on 8 elements, $\nu_1\in \NC[8]$ is non-crossing and $\nu_2\in \PP[8]\setminus\NC[8]$ is crossing. According to Proposition~\ref{prop:WicksWigner}, the mixed moments of large standard Wigner matrices represented by $\nu_1$ is $1$, and the one corresponding to $\nu_2$ is vanishing.\label{fig:2pairings}}
\end{figure}
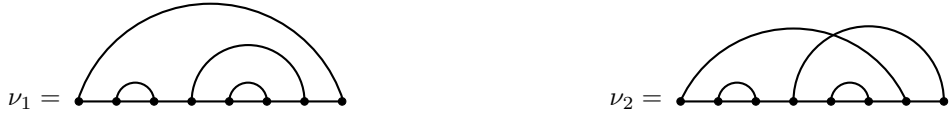

We are now ready to introduce one of the key objects in this survey, a non-commutative analogue of the random matrix model we are interested in $X$. We formulate it below for non-centered and not necessarily self-adjoint matrices, but at times restrict the exposition to the centered and self-adjoint case for simplicity.

\begin{definition}[$\Xfree$]\label{def:Xfree}
   Let $A_0,A_1,\dots,A_n$ be $d\times d$ matrices. Let $W_1^{(N)},\dots,W_n^{(N)}$ be independent standard Wigner matrices, and let $s_1,\dots,s_n$ be a free semicircular family.
\begin{equation}
    X = A_0 + \sum_{k=1}^n g_k A_k, \qquad
    X^{(N)} = A_0\otimes I + \sum_{k=1}^n A_k \otimes W_k^{(N)}, \qquad \Xfree = 
 A_0\otimes \1 + \sum_{k=1}^n A_k \otimes s_k
\end{equation}    
    where $\otimes$ denotes the tensor product ($X^{(N)}$ is an $Nd\times Nd$ matrix), $I$ is a $N
    \times N$ identity, and $\1$ is the identity in the algebra generated by the semicircular family.
\end{definition}

Asymptotic freeness~\eqref{eq:VoiculescuAF} and~\eqref{eq:HT05-SAF} tells us that, as $N\to\infty$, the spectrum of $X^{(N)}$ is well described by the one of $\Xfree$. Respectively, they state that $\lim_{N\to\infty}\EE\tr \big(X^{(N)}\big)^p = (\tr\otimes\tau)\big(\Xfree\big)$ and $\lim_{N\to\infty}\EE \big\|X^{(N)}\big\| = \big\|\Xfree\big\|$.


Another important ingredient is that $\Xfree$ satisfies~\eqref{thm:NCK-Ak} without dimensional factors.

\begin{proposition}\label{prop:Xfreesigma2sigma}[\cite[Theorem 9.9.5]{Pisier2003IntroductionTO}]
Let $X$ be a centered self-adjoint gaussian matrix and $\sigma(X)^2=\|\EE X^2\|$ then
\[
\sigma \leq \big\| \Xfree \big\| \leq 2\sigma.
\]
\end{proposition}

\begin{proofidea}[Proof idea for the upper bound]
    There are several proofs of this estimate, for example it can be directly obtained from Lehner's formula below (Proposition~\ref{prop:LehnerFormula}). 
    Nevertheless, we find it particularly illuminating to recall the proof of Theorem~\ref{thm:NCK-Ak} above and notice that for $\Xfree$ (or $X^{(N)}$ in limit $N\to\infty$) the number of pair partitions $\big|\PP_2[2p]\big|$ in~\eqref{eq:proofNCK-P2} would be replaced by the number of non-crossing pair partitions $\big|\NC_2[2p]\big|$. The number of non-crossing pair partitions are given by the Catalan numbers and so $\big|\NC_2[2p]\big|< 4^p$. This means that the factor $\big((2p)^p\big)^{\frac1{2p}} = \sqrt{2p}$ would be replaced by $\big((4)^p\big)^{\frac1{2p}} = 2$.
\end{proofidea}

The following is a particularly useful estimate:
\begin{lemma}\label{lemma:pisierXfreebound}[Pisier~\cite[\S 9.9]{Pisier2003IntroductionTO}; see also~\cite[Lemma~2.5, \S 4.1]{bandeira2023free})]
    Let $A_0,A_1,\dots,A_n$ be $d\times d$ matrices and $\Xfree$ as in Definition~\ref{def:Xfree}:
\begin{equation}\label{eq:PisierXfreebound}
\frac12\Big(\|A_0\| \maxV \sigma(X)\Big) \leq \big\| \Xfree \big\| \leq \|A_0\| + \left\| \sum_{k=1}^n A_kA_k^\top \right\|^{\frac12} + \left\| \sum_{k=1}^n A_k^\top A_k \right\|^{\frac12}.
\end{equation}
\end{lemma}

In fact, $\|\Xfree\|$ enjoys a remarkable exact formula~\cite{LehnerFormula1999}, that can be written as a semidefinite program~\cite{kunisky-lehner}.

\begin{proposition}\label{prop:LehnerFormula}[Lehner's formula~\cite{LehnerFormula1999} (see also~\cite[Lemma~2.4]{bandeira2023free})]

Let $A_0,A_1,\dots,A_n$ be self-adjoint matrices and $\Xfree$ as in Definition~\ref{def:Xfree}:
\begin{equation}\label{eq:LehnerFormula}
\big\| \Xfree \big\| = \max_{\eps\in\{\pm1\}} \inf_{Z\succeq 0} \lambda_{\max}\left(Z^{-1} + \eps A_0 + \sum_{k=1}^n A_k Z A_k \right).
\end{equation}
\end{proposition}

\begin{remark}[Asymptotic Freeness, Strong Convergence, and Operator Spaces]
The asymptotic freeness phenomenon has important implications in operator algebras~\cite{HAAGERUPetal99,HaagerupThorbjornsen2005,HAAGERUPetal06}, in particular the seminal paper of Haagerup and Thorbj{\o}rnsen~\cite{HaagerupThorbjornsen2005} (that proved~\eqref{eq:HT05-SAF}) settled an important open question in operator algebras. 
In a certain sense, these results show that certain algebras of interest (such as the ones generated by a semicircular family) are well approximated by finite dimensional objects. From the viewpoint of random matrix theory the key consequence is of opposite nature: one can often perform computations directly with the limiting objects, and asymptotic freeness then provides a powerful bridge to transfer such computations to finite (but large) dimensional random matrices.
We take the opportunity to point the reader to a new line of work on establishing strong convergence in a variety of random matrix contexts~\cite{chen2024strongconvergenceI,chen2024strongconvergenceII}, and to an excellent survey on the strong convergence phenomenon by van Handel~\cite{vanHandelStrongConv2025} (which will also be the subject of an ICM talk in 2026).
The intrinsic freeness phenomenon that this survey aims to describe is different: our goal is to study $X=X^{(1)}$ in Definition~\ref{def:Xfree}, and not $X^{(N)}$ as $N\to\infty$. As it turns out, oftentimes $X=X^{(1)}$ already approximates $\Xfree$.
\end{remark}

\subsection{Non-Asymptotic Intrinsic Freeness:}\label{sec:nonasymptoticfreeness}

The key phenomenon that will fuel the sequel (shown by the author, Boedihardjo, and van Handel~\cite{bandeira2023free} and further refined by the author, Cipolloni, Schr\"{o}der, and van Handel~\cite{bandeira2024free2}) is the fact that, in many settings, a gaussian random matrix $X$  behaves like $\Xfree$ without the need to take $\lim_{N\to\infty}X^{(N)}$. We will show that is the case when a certain parameter $v(X)$ is small.

\begin{definition}\label{def:sigmaandv(X)}
Given a $d\times d$ matrix $X$ with jointly gaussian entries (which we will write as $X = A_0 + \sum_{k=1}^n g_kA_k$) we define the following parameters
\begin{equation}
    \sigma(X)^2 = \big\| \EE(X - \EE X)(X - \EE X)^\top \big\| \maxV \big\| \EE(X - \EE X)^\top(X - \EE X) \big\| = \left\| \sum_{k=1}^n A_kA_k^\top \right\| \maxV \left\| \sum_{k=1}^n A_k^\top A_k \right\|;
\end{equation}
\begin{equation}
    v(X)^2 = \big\| \cov(X) \big\| = \sup_{\|B\|_F = 1} \left|\Tr(B^\top A_k)\right|^2,
\end{equation}
where $\cov(X)$ denotes the $d^2\times d^2$ covariance matrix of the entries of $X$;
\begin{equation}
    \sigma_{\ast}(X)^2 = \sup_{u,v\in\mathbb{S}^{d-1}}\EE \big| u^\top (X - \EE X)v\big|^2= \sup_{u,v\in\mathbb{S}^{d-1}}\sum_{k=1}^n\big| u^\top A_kv\big|^2.
\end{equation}
\end{definition}
It is relatively straightforward to see that $\sigma_\ast(X)\leq \sigma(X)\minV v(X)$. The parameter $\sigma_\ast$ corresponds to the Lipschitz constant of $g\to \big\|A_0 + \sum_{k=1}^n g_kA_k\big\|$ and governs the tail estimates when using gaussian concentration to bound $\PP\big(\|X\| \geq \EE\|X\| + t \big)$ for $t>0$ (see~\cite[\S 9]{MDS-Book-2025}).\footnote{The fact that $\sigma_\ast(X)\leq \sigma(X)\minV v(X)$ is essentially the reason why it is generally a good strategy to focus on estimates on $\EE\|X\|$ and then use scalar concentration inequality to obtain tail estimates.}

We are now ready to present the main result, and a brief sketch of its proof.

\begin{theorem}[Intrinsic Freeness~\cite{bandeira2023free,bandeira2024free2}]\label{thm:intrinsicfreeness}

Let $X$ be a $d\times d$ random matrix with jointly gaussian entries (not necessarily centered or self-adjoint), we have
\begin{equation}
    \Big| \EE\|X\| - \|\Xfree\|\Big| \leq C\tilde{v}(X)(\log d)^{\frac34},
\end{equation}
and, for all $t\geq 0$,
\begin{equation}
  \PP\left[ \Big| \EE\|X\| - \|\Xfree\|\Big| > C\tilde{v}(X)(\log d)^{\frac34} + 
    C\sigma_\ast(X)t
   \right] \leq \exp(-t^2),
\end{equation}
where $C$ is a universal constant, $\tilde{v}(X) = \sqrt{v(X)\sigma(X)}$ and $\sigma(X),v(X),\sigma_{\ast}(X)$ are as in Definition~\ref{def:sigmaandv(X)}. Recall that $\sigma(X)\leq\|\Xfree\|\leq2\sigma(X)$.

If $X$ is self-adjoint the same inequalities hold replacing $\|X\|$, $\|\Xfree\|$ by $\lambda_{\max}(X)$, $\lambda_{\max}(\Xfree)$ or by $\lambda_{\min}(X)$, $\lambda_{\min}(\Xfree)$.

For self-adjoint $X$ we also have:
\begin{equation}
    \PP\left[ d_{\mathrm{H}}\left(\mathrm{sp}(X),\mathrm{sp}(\Xfree) \right) > C\tilde{v}(X)(\log d)^{\frac34} + 
    C\sigma_\ast(X)t \right]\leq \exp(-t^2),
\end{equation}
where $\mathrm{sp}(M)$ denotes the spectrum of $M$ and
\[
d_{\mathrm{H}}(A,B) \defeq \inf\left\{ \eps>0:\, A\subseteq B + [-\eps,\eps] \text{ and } B\subseteq A + [-\eps,\eps] \right\} ,
\]
denotes the Hausdorff distance between two subsets of the real line.
\end{theorem}

\begin{remark}[When to use Theorem~\ref{thm:intrinsicfreeness}]
Theorem~\ref{thm:intrinsicfreeness} is useful when $v(X)\ll \sigma(X)/ (\log d)^{\frac32}$ as in that case, since $\sigma(X)\leq\|\Xfree\|\leq 2\sigma(X)$, all terms with a universal constant become negligible. Fortunately, this appears to be fairly common (you can see several applications in~\cite{bandeira2023free,bandeira2024free2}). For example, for standard Wigner matrices we have $\sigma(X)=1$ and $v(X)=\sqrt{\frac2d}$. For $X$ a self-adjoint gaussian random matrix with otherwise independent entries\footnote{The case of independent entries can be studied with other tools~\cite{bandeira2016sharp,latala2018dimension} that allow sparser matrices, we just mention it here for illustrative purposes.} where, for $i\leq j$, $X_{ij}\sim\NNN(0,1)$ if $|i-j|\leq B$ and $X_{ij}=0$ if $|i-j|> B$ we have $\sigma(X) = \sqrt{B}$ and $v(X)=2$, meaning that $v(X)\ll \sigma(X)/ (\log d)^{\frac32}$ as long as $B \gg (\log d)^3$. Another interesting model is that of Pattern Matrices~\cite[\S 3.2.1]{bandeira2023free}, standard Wigner matrices where sets of entries where conditioned to be equal, in the most interesting case in which the patterns of equal entries only have at most one entry per row or column, $v(X)\ll \sigma(X)/ (\log d)^{\frac32}$ holds as long as the largest set of equal entries has size $\ll d/(\log d)^3$. We remark that most often Theorem~\ref{thm:intrinsicfreeness} is used in tandem with Proposition~\ref{prop:LehnerFormula}, Lemma~\ref{lemma:pisierXfreebound}, or simply by using $\sigma(X)\leq \|\Xfree\|\leq 2\sigma(X)$ when $X$ is centered (Proposition~\ref{prop:Xfreesigma2sigma}).
\end{remark}

\begin{proofidea}[Proof sketch of Theorem~\ref{thm:intrinsicfreeness}]
    Let us start by focusing on how to show an upper bound such as $ \EE\|X\|   \leq \|\Xfree\| + C\tilde{v}(X)(\log d)^{\frac34}$ in the self-adjoint case. The key idea in~\cite{bandeira2023free} is to interpolate between $X$ and $\Xfree$. More precisely, this is done via the following random matrix, for $q\in[0,1]$:
    
    \begin{equation}
        X_q^{(N)} = A_0\otimes I + \sqrt{q}\sum_{i=1}^{n}A_i\otimes D_i^{(N)} + \sqrt{1-q}\sum_{i=1}^{n}A_i\otimes W_i^{(N)},
    \end{equation}
    where $I$ is an $N\times N$ identity matrix, $W_1^{(N)},\dots,W_n^{(N)}$ are iid standard 
    Wigner matrices and $D_1^{(N)},\dots,D_n^{(N)}$ are iid $N\times N$ diagonal matrices with iid $\NNN(0,1)$ entries in the diagonal. $X_0^{(N)}=X^{(N)}$ (which, for large $N$, behaves like $\Xfree$). $X_1^{(N)}$ is a $Nd\times Nd$ block diagonal matrix with iid copies of $X$ in the diagonal blocks, in particular $\EE\tr X^p = \EE\tr \big(X_1^{(N)}\big)^p$ for all positive integers $p$.

    The derivative $\frac{d}{dq}\EE \tr \big(X_q^{(N)}\big)^{2p}$ can be computed exactly with Gaussian Interpolation (see~\cite[\S 8]{MDS-Book-2025}) and can be controlled by a matrix alignment parameter\footnote{Intuitively, because the way $\EE \tr \left(X_0^{(N)}\right)^p$ and $\EE \tr \left(X_1^{(N)}\right)^p $ differ are on crossing partitions.} $\tilde{w}(X) = \sup_N w\big(X_1^{(N)}\big)$.
    Furthermore, the alignment parameter can be controlled by the easier to compute quantity
\begin{equation}
    w(X) \leq \sqrt{\sigma(X)v(X)},
\end{equation}
and so $\tilde{w}(X) \leq \sqrt{\sigma\big(X_1^{(N)}\big)v\big(X_1^{(N)}\big)} = \sqrt{\sigma(X)v(X)}.$

After taking $N\to\infty$, this eventually results in the estimate
\begin{equation}\label{eq:tracemomentsXXfree}
    \Big| \left(\EE \tr X^{2p}\right)^{\frac1{2p}} - \left((\tr\otimes \tau) \Xfree^{2p}\right)^{\frac1{2p}}\Big| \leq 2 p^{\frac34} \tilde{v}(X).
\end{equation}

For $p\sim\log(d)$, $\left(\tr X^{2p}\right)^{\frac1{2p}}$ captures the spectrum of $X$, since $\left(\tr X^{2p}\right)^{\frac1{2p}} \leq \|X\|\leq d^{\frac1{2p}}\left(\tr X^{2p}\right)^{\frac1{2p}}$.\footnote{For any $\eps>0$, the inequality \eqref{eq:tracemomentsXXfree} for $d=\lfloor C_\eps' \log d\rfloor$ would give $ \EE\|X\| \leq (1+\eps)\|\Xfree\| + C_\eps \tilde{v}(X)(\log d)^{\frac34}$ for $C_\eps,C_\eps'$ constants depending on $\eps$, but it is possible to obtain the sharp leading order term~\cite{bandeira2023free}.}
To obtain information about $\mathrm{sp}(X)$ (and not just trace moments),~\cite{bandeira2023free} interpolates other spectral statistics, in particular moments of the resolvent $\EE \tr \big|zI - X\big|^{-2p}$.
Non self-adjoint matrices can be handled by Hermitian dilation (Remark~\ref{rem:HermitianDilation}), while tail bounds can be obtained by scalar concentration of measure (such as Gaussian concentration~\cite[\S8]{MDS-Book-2025}).

The argument above shows that spectral statistics of $X$ and $\Xfree$ are close and that spectral statistics of $X$ (such as $\EE\left(\tr X^{2p}\right)^{1/2p}$ for $p\sim\log d$) can describe the spectrum of $X$. On the other hand, $\Xfree$ is an infinite dimensional operator, it is not a priori clear that spectral statistics such as $\big(\tau( \Xfree^{2p})\big)^{1/2p}$ capture the behavior of the spectrum of $\Xfree$ for the values of $p$ the argument can handle (see Figure~\ref{fig:bad}). The main technical contribution of~\cite{bandeira2024free2} is to show that this is indeed the case, it can be viewed as a regularity guarantee for $\Xfree$, showing that the spectrum is sufficiently regular so that the spectral statistics interpolated in the argument capture $\mathrm{sp}(\Xfree)$.\footnote{As we will see in Section~\ref{sec:sharptransitions}, the fact that estimates are two sided will allows us to capture important phase transitions in problems arising in Theoretical Computer Science and Statistics that would have been impossible to capture with upper bounds alone. For example, when studying a random matrix corresponding to a spectral method in statistics or computer science, Theorem~\ref{thm:intrinsicfreeness} allows us not only to control the effects of noise, but also to show that the signal of interest is indeed visible in the spectrum (see Section~\ref{sec:sharptransitions} and Figure~\ref{fig:bbp}).}
\end{proofidea}

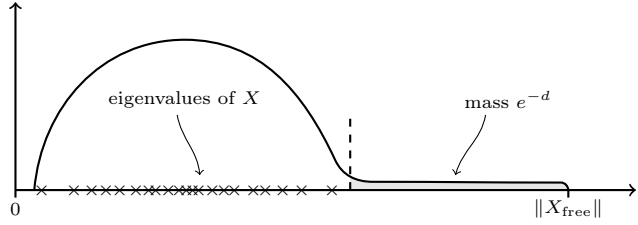
\begin{figure}[h]
\centering
\begin{tikzpicture}

\begin{scope}
\clip (4.2,0) rectangle (7.2,1);
\fill[gray!20]
(4,0) to
(4,0.4)
to[out=295,in=180] (4.5,0.11)
to[out=0,in=180] (7,0.1) to[out=0,in=95] (7.1,0)
to (4,0); 
\end{scope}

\draw[thick] (0,0) to[out=85,in=180] (2,2) 
to[out=0,in=115] (4,0.4)
to[out=295,in=180] (4.5,0.11)
to[out=0,in=180] (7,0.1) to[out=0,in=95] (7.1,0);

\draw[thick,->] (-.25,0) -- (8,0);
\draw[thick,->] (-.25,0) -- (-.25,2.5);

\draw[thick] (7.1,0) -- (7.1,-.1);
\draw[thick] (-.25,0) -- (-.25,-.1);
\draw (7.1,-.25) node {$\scriptstyle\|X_{\rm free}\|$};
\draw (-.25,-.25) node {$\scriptstyle0$};

\draw[thick,dashed] (4.2,0) -- (4.2,1);

\draw[->] (6,1) to[out=260,in=75] (5.6,0.2);
\draw (6.3,1.2) node {$\scriptstyle\text{mass }e^{-d}$};

\draw[->] (1.9,1) to[out=280,in=95] (2.2,0.2);
\draw (2,1.2) node {$\scriptstyle\text{eigenvalues of }X$};

\pgfmathsetseed{701}

\foreach \i in {0,...,10}
{
	\draw ({2*(\i/10)^(2/3)+rnd/10},0) node
	{$\scriptstyle\times$};
}
\foreach \i in {0,...,9}
{
	\draw ({4-2*(\i/10)^(2/3)-rnd/10},0) node
	{$\scriptstyle\times$};
}

\end{tikzpicture}
\caption{Illustration of a hypothetical obstruction to the validity of 
Theorem~\ref{thm:intrinsicfreeness}, where the spectral statistics used in the interpolation argument do not capture whole of the spectrum of $\|\Xfree\|$. The main result in~\cite{bandeira2024free2} can be viewed as a regularity guarantee for the spectrum of $\Xfree$ that, in particular, rules out this situation.\label{fig:bad}}
\end{figure}

\subsubsection{Universality:}\label{sec:universality}

Recently, Brailovskaya and van Handel~\cite{Tatianaetal_Universality} developed a universality principle to handle random matrices of the form
\begin{equation}\label{eq:Y_0+allYs}
    Y = Y_0+\sum_{i=1}^n Y_i,
\end{equation}
with $Y_0$ deterministic and $Y_1,\dots,Y_n$ independent and centered. They showed that $Y$ as in~\eqref{eq:Y_0+allYs} behaves, as long as all summands are sufficiently small, like a gaussian analogue $\Ygauss$ where the entries of $Y$ are replaced by gaussian random variables with the same mean and covariance.\footnote{Note that this is different from symmetrization where $Y =\sum_{i=1}^n Y_i$ would be analysed via $\sum_{i=1}^ng_i Y_i$.} The matrix $\Ygauss$ can then often be handled with the tools of intrinsic freeness.

\begin{theorem}
[\cite{Tatianaetal_Universality}]\label{thm:universalityspectrum}
    Let $Y$ be a $d\times d$ self-adjoint random matrix as in~\eqref{eq:Y_0+allYs} with $\|Y_i\|\leq R$ almost surely for all $1\leq i\leq n$. Let $\Ygauss$ be the gaussian matrix whose entries have the same mean and covariance, then
    \begin{equation}
    \PP\left[ d_{\mathrm{H}}\left(\mathrm{sp}(Y),\mathrm{sp}(\Ygauss) \right) > C\sigma_{\ast}(X)t^{\frac12}+ CR^{\frac12}\sigma(X)^{\frac23}t^{\frac23} + CRt \right]\leq d\exp(-t),
\end{equation}
for all $t\geq 0$, where $C$ is a universal constant.
\end{theorem}

Combining these tools with Theorem~\ref{thm:intrinsicfreeness} one obtains easy to use improvements of the matrix Bernstein inequality.

\begin{theorem}[\cite{bandeira2023free,Tatianaetal_Universality}] \label{thm:improvedBernstein}
Let $Y_1,\dots,Y_n\in\RR^{d\times d}$ be random independent symmetric matrices satisfying $\EE Y_i=0$, and such that $\|Y_i\|\leq R$, for all $i\in[n]$, almost surely. Then
\[
\EE \left\| \sum_{i=1}^n Y_i  \right\| \leq 2\sigma + C \Big( v^{\frac12}\sigma^{\frac12}(\log d)^\frac34 + R^{\frac13}\sigma^{\frac23}(\log d)^\frac23 + R \log d \Big),
\]
and
\[
\PP\left[  \left\| \sum_{i=1}^n Y_i  \right\| \geq 2\sigma + C \Big( v^{\frac12}\sigma^{\frac12}(\log d)^\frac34 + \sigma_\ast t^{\frac12}+ R^{\frac13}\sigma^{\frac23}t^\frac23 + R t \Big)    \right] \leq de^{-t}
\]
where, $C$ is a universal constant, 
\begin{equation}
\sigma^2 = \left\|  \sum_{i=1}^n \EE Y_i^2   \right\| \text{, } v^2 = \left\| \cov(Y) \right\|, R=\big\| \max_i \|Y_i\| \big\|_{\infty} \text{ and }, \sigma_\ast^2 = \sup_{\|u\|=\|w\|=1}\EE \left| v^TYw \right|^2.
\end{equation}
\end{theorem}

Note that if $v,\sigma_\ast,R\ll \sigma / \polylog(d)$, which happens often in applications (see~\cite{bandeira2023free,Tatianaetal_Universality,bandeira2024free2}), then all terms multiplying the universal constant $C$ are negligible and the tail parameter $t$ appears only in low-order terms.

\section{Some Extensions and Applications}\label{sec:extensionsapplications}

In this section we briefly describe some extensions and applications of the intrinsic freeness phenomenon, focusing on applications that showcase how easy it is to use these methods, in particular in problems in high dimensional statistical estimation and theoretical computer science. Due to space constraints, the descriptions will be at a bird's-eye-view level and refer to the original references for more information.

\subsection{Sharp Phase Transitions:}\label{sec:sharptransitions}

There are many applications in statistics and theoretical computer science where the central object is a random matrix $X(\lambda) = \lambda Z_0 + Z$ where $Z_0$ is deterministic and corresponds to a signal of interest, $\lambda\geq0$ represents the signal-to-noise ratio (SNR) and $Z$ is a centered random matrix representing noise (or other types of data corruption). In this setting, it is usual that success of an algorithm of interest corresponds to whether $\EE X(\lambda)=\lambda Z_0$ is visible in the spectrum of $X$, or whether it is drown out by the noise $Z$.\footnote{Usually it is also important that the leading eigenvectors (or singular vectors) of $Z$ correlate with --have significant inner-product with-- leading eigenvectors (or singular vectors) of $Z_0$, often refered to as \emph{eigenvector overlap}. This tends to happen at precisely the same critical value of $\lambda$ as when the spectral norm (or leading eigenvalue) of $Z(\lambda)$ has a phase transition, so we will focus our exposition on extremal eigenvalues. The references we cite for each result also address eigenvector overlap.}

Armed with a good upper bound on $\EE\|Z\|$ (and the fact that $\|Z\|= (1\pm o(1))\EE\|Z\|$ with high probability, which usually follows from scalar concentration\footnote{In this exposition we mostly focus on bounds of expectations $\EE\|\cdot\|$ because norms of a random matrices usually concentrate significantly and standard scalar concentration techniques tend to show that, with high probability, their deviations with respect to their mean, $\EE\|\cdot\|$, are lower order. The references cited include all the precise tail bound estimates.}) one can readily obtain a lower bound on the critical level of SNR $\lambda$ for which $\lambda Z_0$ is visible in the spectrum. By Jensen's inequality, $E\| X(\lambda)\|\geq \lambda\|Z_0\|$. Usually this can be transformed into a guarantee that if, for some $\eps>0$,
\begin{equation}\label{eq:triangularineqBBP}
\lambda\geq (1+\eps)\frac{1}{\|Z_0\|}\EE\|Z\|,    
\end{equation} then the signal is visible in the spectrum, in the sense that there exists a threshold $T$ for which $\|X(\lambda)\|>T$ and $\|X(0)\|<T$ with high probability.

Unfortunately, arguments of this nature tend to be suboptimal regardless of how sharp the bounds on $\EE\|Z\|$ are. Let us describe a classical example, the spiked Wigner matrix model and the celebrated BBP transition~\cite{Johnstone2001,bbp2005,feralpeche2007}. Let $v\in\mathcal{S}^{d-1}$, the spiked Wigner matrix model is the $d\times d$ random matrix
\begin{equation}\label{eq:spikedwignerdef}
    X(\lambda) = \lambda vv^\top + W,
\end{equation}
where $W$ is a standard Wigner matrix. Since $\EE\|W\|=2(1\pm o(1))$ and $\|vv^\top\|=1$, the argument above can only guarantee that a perturbation on the spectrum of $X(\lambda)$ is visible for $\lambda \geq 2+\eps$. However, it is known that this transition happens at $\lambda=1$. To be more precise, let us define
\begin{equation}\label{eq:def:Blambda}
    B(\lambda)\defeq \begin{cases}
        2 & \text{for } \lambda \leq 1 \\
        \lambda + \frac{1}{\lambda} & \text{for } \lambda > 1.
    \end{cases}
\end{equation}
It is well known~\cite{feralpeche2007} that the largest eigenvalue of $X(\lambda)$ converges to $B(\lambda)$, showing that the phase transition happens at the critical threshold $\lambda=1$.\footnote{Moreover, this is statistical optimal in the sense that as long as the prior on $v$ is sufficiently uninformative, no statistical procedure can succeed for $\lambda<1$~\cite{bandeiraetallBBPlowerboundsAoS2018}.}

While these sharp phase transitions were only characterized for very specific random matrix ensembles (usually with i.i.d. entries, or enjoying rotational symmetry), an important consequence of the two sided bounds in Theorem~\ref{thm:intrinsicfreeness} is the fact that they allow to establish this type of sharp phase transitions in essentially any random matrix model for which Theorem~\ref{thm:intrinsicfreeness} can be used, potentially in tandem with the universality principle in Theorem~\ref{thm:universalityspectrum} (and where $\|\Xfree\|$ can be computed). It is a remarkable fact that $\Xfree$ is able to ``witness'' a low-rank perturbation in $X$ (see Figure~\ref{fig:bbp}).

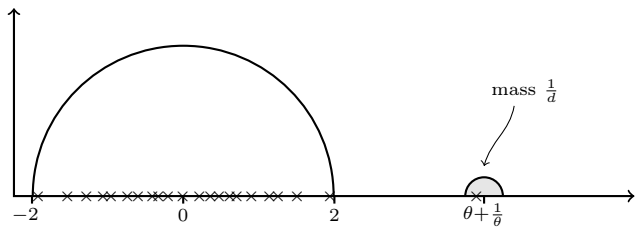
\begin{figure}[h]
\centering
\begin{tikzpicture}

\begin{scope}
\clip (-.25,0) rectangle (7.2,2.1);

\fill[gray!20] (6,0) circle (0.25);
\draw[thick] (2,0) circle (2);
\draw[thick] (6,0) circle (0.25);
\end{scope}

\draw[thick,<->] (8,0) -- (-.25,0) -- (-.25,2.5);

\pgfmathsetseed{671}

\foreach \i in {0,...,10}
{
	\draw ({1.9*(\i/10)^(2/3)+rnd/10},0) node
	{$\scriptstyle\times$};
}
\foreach \i in {0,...,9}
{
	\draw ({4-1.9*(\i/10)^(2/3)-rnd/10},0) node
	{$\scriptstyle\times$};
}

\draw (5.9,0) node {$\scriptstyle\times$};

\draw[thick] (6,0) -- (6,-.1);
\draw[thick] (2,0) -- (2,-.1);
\draw[thick] (-.01,0) -- (-.01,-.1);
\draw[thick] (4.01,0) -- (4.01,-.1);
\draw (6,-.25) node {$\scriptstyle \theta+\frac{1}{\theta}$};
\draw (-.1,-.25) node {$\scriptstyle-2$};
\draw (2,-.25) node {$\scriptstyle0$};
\draw (4.01,-.25) node {$\scriptstyle2$};

\draw[->] (6.4,1.2) to[out=260,in=75] (6,0.4);
\draw (6.55,1.4) node {$\scriptstyle\text{mass }\frac{1}{d}$};

\end{tikzpicture}
\caption{Illustration of how Theorem~\ref{thm:intrinsicfreeness} can capture the celebrated BBP transition~\cite{feralpeche2007,bbp2005} in the spiked Wigner model: $X(\lambda) = \lambda vv^\top + W$, where $W$ is a standard Wigner matrix and $v\in\mathbb{S}^{d-1}$ is fixed. Even though $\|\lambda vv^\top\|>\|W\|$ would require $\lambda>2$, it is known that the largest eigenvalue of $X(\lambda)$ enjoys a phase transition at $\lambda=1$. This phenomenon is visible on the spectrum of $\Xfree(Y)$, depicted here (the semi-circles depict the spectrum of $\Xfree$ and the $\times$'s that of a draw of the spiked Wigner matrix model. This phenomenon illustrates how we are making use of free probability in a non-asymptotic way, as if we took the asymptotic limit $d\to\infty$ the rank-1 perturbation would not be visible in the weak convergence of the spectrum. \label{fig:bbp}}
\end{figure}

A particularly elegant class of examples is the isotropic case, where $\EE(X-\EE X)^2= \sigma(X)^2 I$.

\begin{theorem}[\cite{bandeira2024free2}]\label{thm:isotropicbbpXfree}
    Let $X$ be a $d\times d$ self-adjoint gaussian random matrix for which $\EE(X-\EE X)^2= \sigma(X)^2 I$ and for which $\EE X$ has rank $r$. If $\sigma_\ast(X)\sqrt{r}\leq 1$ then
\begin{equation}
    \big| \lambda_{\max}(\Xfree) - B\big( \lambda_{\max}(\EE X) \big)  \big| \leq 2\sigma_{\ast}(X) \sqrt{r},
\end{equation}
where $B(\lambda)$ is given by~\eqref{eq:def:Blambda}.
\end{theorem}

Combined with Theorem~\ref{thm:intrinsicfreeness} it guarantees that, as long as $\tilde{v}(X)(\log d)^{\frac34} \maxV \sigma_\ast(X) \sqrt{r} \ll \sigma(X)$, then \[\lambda_{\max}(X)=(1\pm o(1))B\big( \lambda_{\max}(\EE X) \big),\] with high probability. This result goes significantly beyond the classical spiked Wigner matrix model (for which $\sigma_\ast(X)=2/\sqrt{d}$), allowing both high rank perturbations and random matrices that do not have iid entries (such as sparse matrices~\cite{bandeira2024free2} and the Kikuchi matrix example in Section~\ref{sec:kikuchi}), and non-gaussian matrices (see~\cite{bandeira2024free2}). 

These tools also allow one to characterize phase transitions in non isotropic random matrix models. A notable example is a random matrix arising in an algorithm to do signal recovery in an inhomogeneous spike model of Pak, Ko, and Krzakala~\cite{PakKoKrzakala-Conjecture2023} (where the standard spectral method is information-theoretically suboptimal~\cite{GKKZ-InfThLB-22}). In~\cite{PakKoKrzakala-Conjecture2023} a phase transition is conjectured at a particular threshold predicted with statistical physics tools. Using Theorem~\ref{thm:intrinsicfreeness},~\cite{bandeira2024free2} proved this conjecture.\footnote{This conjecture was also concurrently proven in~\cite{MergnyKoKrzakala-Solution2024}, although the proof using Theorem~\ref{thm:intrinsicfreeness} provides a stronger version where the guarantees are non-asymptotic and certain important parameters are allowed to depend on the size of the matrix.} A related example is the characterization of the critical SNR at which the spectral method in~\cite{Deshpande-cSBM2018} perform detection in the Contextual Stochastic Block Model (see~\cite{bandeira2024free2}). We refer to~\cite{bandeira2024free2} for more details and for descriptions of more applications.




\subsection{Matrix Chaos and Iterated Matrix Concentration:}\label{sec:matrixchaos}

A remarkable feature of the non-commutative Khintchine inequality is that it can be iterated~\cite[Remark 9.8.9]{Pisier2003IntroductionTO}, allowing to handle important classes of random matrices beyond the ones normally handled with matrix concentration tools, matrix chaoses. 
The approach of iterating inequalities such as~\eqref{eq:nck1} dates back, in the area of operator spaces, to~\cite{Haagerup1993BoundedLO} and (special cases) have been reinvented a several times in the literature in different application contexts (see~\cite{bandeira2025matrixchaos}). 

A gaussian \emph{matrix chaos} of order $q$ is the following random matrix model
\begin{equation}\label{eq:chaos}
    X = \sum_{\substack{i_1,\ldots,i_q\in[n] \\ i_1,\ldots,i_q \text{ distinct}}} g_{i_1}\cdots g_{i_q} A_{i_1,\dots,i_q},
\end{equation}
where $g_1,\dots,g_n$ are i.i.d. standard Gaussian, and
$A_{i_1,\ldots,i_q}$ are deterministic $d_1\times d_2$ matrix coefficients. We will represent the collection of matrix coefficients as a $q+2$ order tensor $\AAA$ where the first $q$ coordinates correspond to the chaos coordinates, and the last two to the matrix coordinates ($\AAA(i_1,\ldots,i_q,s,t)=(A_{i_1,\dots,i_q})_{s,t}$). 

The inequalities we will describe below are defined in terms of the norms of \emph{flattenings} of the tensor $\AAA$ that are defined as follows. Denote by $e_i$ the $i$th element of the standard coordinate basis, viewed as a column vector. Then for any subsets $R,C\subseteq[q+2]$, we define
the matrix
\begin{equation}\label{eq:flattenings}
    \flatta{R}{C} := \sum_{\substack{i_1, \ldots, i_q \in [m]\\i_{q+1} \in [d_1], i_{q+2} \in [d_2]}} \left(\bigotimes_{t \in R} e_{i_t}\right) \otimes \left(\bigotimes_{t \in C} e_{i_t}^\top\right) \AAA_{i_1, \ldots, i_{q+2}},
\end{equation}
where $\bigotimes$ denotes tensor product. 
This definition is easiest to interpret when $R=[q+2]\backslash C$: in this case, $\flatta{R}{C}$ is the matrix whose rows are indexed by the coordinates in the row set $R$, whose columns are indexed by the coordinates in the column set $C$, and whose entries are the corresponding entries of $\AAA$. For example, if $q=2$ and $R=\{1,3\}$, $C=\{2,4\}$, then the associated flattening $\flatta{R}{C}$ is the $md_1\times md_2$ matrix with entries
$(\flatta{R}{C})_{(i_1,i_3),(i_2,i_4)} = \AAA_{i_1,i_2,i_3,i_4}$. 

For sake of exposition, we will focus our description of the iteration procedure to $q=2$. Let $X = \sum_{i\neq j \in [n]} g_i g_j A_{ij}$. The first step is to use classical decoupling inequalities~\cite[Theorem 3.1.1]{Gin1998DecouplingFD} to show that, for $C_q$ a constant which depends only on the degree $q$ (and so in this particular case is universal), $\EE\|X\|\leq C_q \EE\|Y\|$ for $Y = \sum_{i\neq j \in [n]} g^{(1)}_i g^{(2)}_j A_{ij}$ where $g^{(1)}_1,\dots,g^{(1)}_n,g^{(2)}_1,\dots,g^{(2)}_n$ are i.i.d. standard Gaussians. For decoupled chaoses the square-free condition can be dropped, so we will focus on understanding 
\(
    \EE\big\| \sum_{i,j \in [n]} g^{(1)}_i g^{(2)}_j A_{ij} \big\|.
\)

It is useful to rewrite the parameter $\sigma$ in 
Theorem~\ref{thm:NCK-Ak}. Note that
\[
\left\| \sum_{k=1}^n A_k^\top A_k \right\|^\frac12 = \left\|\begin{bmatrix} A_1 \\ \vdots \\ A_n\end{bmatrix}^\top
\begin{bmatrix} A_1 \\ \vdots \\ A_n\end{bmatrix}\right\|^\frac12 = \left\|\begin{bmatrix} A_1 \\ \vdots \\ A_n\end{bmatrix} \right\| =    \left\| \sum_{i\in[n]} e_i\otimes A_i \right\|,
    \]
and, analogously,
\(
\left\| \sum_{k=1}^n A_k A_k^\top \right\|^\frac12
=  \left\|  \begin{bmatrix} A_1 & \cdots & A_n\end{bmatrix}\right\|
 = \left\| \sum_{i\in[n]} e_i^\top\otimes A_i\right\|, \)
which means we can write $\sigma = \left\| \sum_{i\in[n]} e_i\otimes A_i \right\| \maxV \left\| \sum_{i\in[n]} e_i^\top\otimes A_i\right\|.$
This and the tower property of the expectation allows us to iterate~\eqref{eq:nck1}:
\begin{eqnarray*}
\EE\left\| \sum_{i \in [n]} g^{(1)}_i \left( \sum_{j\in [n]} g^{(2)}_j A_{ij}\right) \right\|
&\lesssim& \sqrt{\log d}\ 
\EE_{g^{(1)}} \left(\,  \left\| \sum_{j\in[n]} e_j\otimes \left( \sum_{i \in [n]} g^{(1)}_i A_{ij}\right) \right\| \maxV \left\| \sum_{j\in[n]} e_j^\top\otimes \left( \sum_{i \in [n]} g^{(1)}_i A_{ij}\right)\right\| \,\right)\\
&\lesssim& \sqrt{\log d}\ \EE_{g^{(1)}}
 \left(\,  \left\| \sum_{i\in[n]}g^{(1)}_i \left( \sum_{j \in [n]} e_j \otimes A_{ij}\right)  \right\| \maxV\left\| \sum_{i\in[n]}g^{(1)}_i \left( \sum_{j \in [n]} e_j^\top \otimes A_{ij}\right)  \right\| \,\right),
\end{eqnarray*}
where we used~\eqref{eq:nck1} on $g^{(2)}$ and $\EE_{g^{(1)}}$ denotes expectation with respect to $g^{(1)}$. Each of the terms in the right-hand-side is the norm of gaussian matrix, so we can use~\eqref{eq:nck1} again (together with the fact that expectation of the maximum is smaller than sum of expectations) and obtain
\[
\EE\left\| \sum_{i,j \in [n]} g^{(1)}_i g^{(2)}_j A_{ij} \right\| \lesssim (\log d) \ \Big( \left\| \flatta{\{1,2,3\}}{\{4\}} \right\| \maxV\ \left\|
\flatta{\{1,3\}}{\{2,4\}}\right\| \maxV\ \left\|
\flatta{\{2,3\}}{\{1,4\}}\right\| \maxV\ \left\|
\flatta{\{3\}}{\{1,2,4\}}\right\|  \Big).
\]
In general, this process can be iterated $q$ times (see~\cite{bandeira2025matrixchaos}) and yields $\EE\|X\| \lesssim (\log d)^{\frac{q}2} \sigma(\AAA)$ where
\begin{equation}\label{eq:sigmadefinition}
    \sigma(\AAA) \defeq \max_{\substack{R=[q+2]\backslash C\\q+1\in R, q+2\in C}} \left\|\flatta{R}{C}\right\|.
\end{equation}


In~\cite{bandeira2025matrixchaos}, the author, Lucca, Nizi\'{c}-Nikolac, and van Handel realized that the inequalities in Theorem~\ref{thm:intrinsicfreeness} can also be iterated. The key insight is that the parameter $v(X)$ also corresponds to a flattening
\[
v\left( \sum_{i\in[n]}g_iA_i\right) = \left\| \cov\left( \sum_{i\in[n]}g_iA_i\right) \right\|^{\frac12} =\left\| \sum_{i\in[m]} e_i\otimes \vecf(A_i)^\top \right\| =
    \left\|\begin{bmatrix}\vecf(A_1)^\top \\ \vdots \\ \vecf(A_m)^\top\end{bmatrix}
    \right\|,
\]
where the original matrix dimensions are both taken to column indices. This allows~\cite{bandeira2025matrixchaos} to obtain inequalities for the spectral norm of matrix chaoses where the dimension dependency appears in a potentially negligible term.

\begin{theorem}[\cite{bandeira2025matrixchaos}]\label{thm:iteratedfreenck}
    Let $X$ be a matrix chaos as in \eqref{eq:chaos}. Then
    \begin{equation*}
        \EE\|Y\| \lesssim_q 
        \left(\sigma(\AAA) + \log(d_1+d_2+m)^{\frac{q+2}{2}} v(\AAA)\right),
    \end{equation*}
    where 
\begin{equation}\label{eq:vdefinition}
    v(\AAA) \coloneqq \max_{\substack{R= [q+2]\backslash C\\q+1, q+2\in C\\R \neq \emptyset}} \left\|\flatta{R}{C}\right\|.
\end{equation}
\end{theorem}

The lower bound in~\eqref{eq:nck1} can also be iterated to show that these inequalities are essentially tight~\cite{bandeira2025matrixchaos}. Moreover, it is also possible to iterated matrix Rosenthal inequalities to obtain non-gaussian versions of these inequalities~\cite{bandeira2025matrixchaos}.
Furthermore,  
there is a large class of matrix chaoses, called \emph{of combinatorial type} where computing~\eqref{eq:sigmadefinition} and~\eqref{eq:vdefinition} amounts to a simple exercise (see~\cite{bandeira2025matrixchaos}), this includes, among others, the important example of Khatri-Rao random matrices~\cite{KhatriRao1968SoTS} (which appear in numerical linear algebra~\cite{CEMT25:Faster-Linear}) and the example briefly described in Section~\ref{sec:soschaos}.


\subsection{An Illustrative Application: The Tensor PCA Problem:}\label{sec:tensorPCA}

In this section we briefly describe a couple of illustrative applications of the inequalities above in a problem in high dimensional estimation, the Tensor Principal Component Analysis model~\cite{MontanariRichard-TensorPCA2014}.\footnote{The tensor PCA model can be viewed as a tensor version of the matrix spiked models discussed above.}
Here, we will consider a symmetric version of the problem in which the signal of interest is a point in the hypercube. Given $n,r$ and $\lambda$ (we will consider $r$ fixed and $n$ as very large), the goal is to estimate (or detect) an unknown ``signal'' $x\in\{\pm1\}^n$ (drawn uniformly from the hypercube), from ``measurements'' as follows: for $i_1<i_2<...<i_r$, 
\begin{equation}\label{eq:tensorPCA}
Y_{i_1,i_2,...,i_r} = \lambda \big(x^{\otimes r}\big)_{i_1,i_2,...,i_r} + Z_{i_1,i_2,...,i_r},
\end{equation}
where $Z_{i_1,i_2,...,i_r}$ are i.i.d. standard Gaussian (and independent from $x$). Note that $\big(x^{\otimes r}\big)_{i_1,i_2,...,i_r}=\prod_{j=1}^n x_{i_j}$.

Tensor PCA is believed to undergo a fascinating statistical-to-computational gap: without regards for computational efficiency: it is possible to estimate (or detect) $x$ for $\lambda=\tilde{\Omega}\left(n^{-(r-1)/2}\right)$; efficient algorithms, such as the Sum-of-Squares (SOS) hierarchy, are able to solve the problem at $\lambda=\tilde{\Omega}\left(n^{-r/4}\right)$; and local methods, such as gradient descent and approximate message passing succeed at $\lambda=\tilde{\Omega}\left(n^{-1/2}\right)$. Here $\tilde{\Omega}(\cdot)$ may hide constants depending on $r$ and polylogarithmic factors on $n$. Furthermore, it is conjectured that no efficient algorithm can significantly outperform the SOS threshold, giving rise to a statistical-to-computational gap (see~\cite{KuniskyWeinBandeira2022LowDegreeSurvey}). For $r=2$, the problem reduces to a matrix model and all these thresholds coincide. We point the reader to~\cite{Hopkins2015TensorPC,Hopkins2018StatisticalIA,wein2019kikuchi,KuniskyWeinBandeira2022LowDegreeSurvey} and references therein for more on each of these thresholds (see Figure~\ref{fig:TensorPCA-regimes}).

\begin{figure}[h]
    \centering
\begin{tikzpicture}[>=stealth, thick]

\draw[->] (0,0) -- (10,0);

\foreach \x in {2.5,5,7.5} {
  \draw[dashed] (\x,0.3) -- (\x,-0.3);
}

\node at (1.25,0.4) {Impossible};
\node at (3.75,0.4) {Hard};
\node at (6.25,0.4) {Easy};
\node at (8.75,0.4) {Local};

\node[below] at (2.75,-0.1) {$\lambda\sim n^{-\frac{r-1}{2}}
$};
\node[below] at (5.25,-0.15) {$\lambda\sim n^{-\frac{r}{4}}
$};
\node[below] at (7.75,-0.1) {$\lambda\sim n^{-\frac{1}{2}}
$};

\end{tikzpicture}
\caption{The conjectured statistical-to-computational gap in Tensor PCA~\eqref{eq:tensorPCA}~\cite{Hopkins2018StatisticalIA,wein2019kikuchi,KuniskyWeinBandeira2022LowDegreeSurvey}.
    \label{fig:TensorPCA-regimes}}
\end{figure}
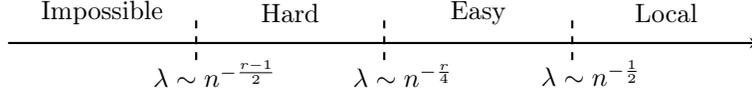

\subsubsection{Kikuchi Matrices:}\label{sec:kikuchi}

A particularly elegant algorithmic approach to tensor PCA, based on the so-called Kikuchi free energy, is due to Wein, El Alaoui, and Moore~\cite{wein2019kikuchi}. It can be viewed as a hierarchy of message passing algorithms that match the performance of the Sum-of-Squares approach, closing an important previously existing gap. We will describe here a spectral method arising from this approach~\cite{wein2019kikuchi}, based on the construction of \emph{Kikuchi Matrices}.\footnote{Kikuchi matrices, and estimates on norms of random Kikuchi matrices, have since been used to make substantial progress in important questions in combinatorics~\cite{Kotharietal-Kikuchi-Moore1-2022,HsiehKothariMohanty-Kikuchi-Moore2-2023} and in the study of locally decodable codes~\cite{Kotharietal-Kikuchi-lowerboundsLDC2023}.} For even $r$, and $\ell\in\NN$ a design parameter (with $\frac{r}2\leq \ell \ll n$), the Kikuchi matrix $M$ is the ${n \choose \ell} \times {n \choose \ell}$ matrix, whose rows and columns are indexed by $\ell$-sized subsets of $[n]$, given by

    \[
    M(\lambda)_{I,J} = \left\{ \begin{array}{ccl}
    Y_{I \Delta J} & \text{if} & |I \Delta J| = r,\\
    0 &\text{otherwise,}&
    \end{array} \right.
    \]
    where $I\Delta J = (I\cup J) \setminus (I\cap J)$ denotes the symmetric difference, and $Y$ is given by~\eqref{eq:tensorPCA}.

    The goal is to understand for which values of $\lambda$ the rank-1 spike in~\eqref{eq:tensorPCA} is visible in the leading eigenvalue of $M(\lambda)$. By symmetry we can assume, without loss of generality that $x_i=1,\ \forall_i$. In that case, $\EE M(\lambda) = \lambda A_0$ where $A_0$ is the adjacency matrix of a graph.\footnote{This graph is tightly connected to Johnson association schemes and so its spectrum can be computed, see~\cite{bandeira2024free2}.} A computation shows that this graph is $d_\ell$-regular with $d_\ell = {\ell \choose r/2}{n-\ell \choose r/2}$. Also, $\EE\big(M(\lambda)-\EE M(\lambda)\big)^2 = d_\ell\, I$. Since $A_0$ can be well approximated by a low rank matrix~\cite{bandeira2024free2}, Theorems~\ref{thm:intrinsicfreeness} and~\ref{thm:isotropicbbpXfree} can be readily used to establish the exact critical threshold for the success of this spectral method at $\lambda=\frac{1}{\sqrt{d_\ell}}\sim_r n^{-r/4}$ (for $\frac{r}2\leq  \ell\leq \frac{3r}4$, which guarantees $\tilde{v}(X)(\log d)^{\frac34} \maxV \sigma_\ast(X) \sqrt{r} \ll \sigma(X)$, where $r$ is the rank of the low-rank approximation of $A_0$). 
    Previously thresholds were only known up to logarithmic factors~\cite{wein2019kikuchi}.\footnote{See also Conjecture 9 in~\cite{RandomstrasseProblems2024}, related to the case of larger $\ell$.}

\subsubsection{Sum-of-Squares and Matrix Chaos:}\label{sec:soschaos}

Countless problems in high dimensional statistics and theoretical computer science can be written as systems of polynomial equalities and inequalities. The Sum-of-squares (SOS) hierarchy of algorithms provides a unified framework to develop algorithms to solve these problems, with a design parameter (the degree) where higher degree versions of SOS provide ever more powerful, but more computationally costly, algorithms. Remarkably, understanding for which parameter regimes problems can be solved with constant level SOS tends to render accurate statistical-to-computational gap predictions, such as the ones in Figure~\ref{fig:TensorPCA-regimes} (see Raghavendra, Schramm, and Steurer's ICM 2018 survey~\cite{Steurer-etal-ICM2018}\footnote{Sum-of-squares is also tightly connected to the low degree method for computational thresholds~\cite{Hopkins2018StatisticalIA,KuniskyWeinBandeira2022LowDegreeSurvey}}).

A particularly elegant argument in this line of work is Hopkins, Shi, Steurer's~\cite{Hopkins2015TensorPC,Hopkins2018StatisticalIA} proof that SOS of degree 6 solves the tensor PCA problem for $\lambda=\tilde{\Omega}(n^{-3/4})$. We will briefly describe how Theorem~\ref{thm:iteratedfreenck} can sharpen the analysis and remove the spurious logarithmic factor.\footnote{Interestingly, in the context of studying Quantum expanders, Lancien and Youssef~\cite{Lancien2023ANO} have also provided an estimate without spurious logarithmic factors for the same random matrix chaos, using Theorems~\ref{thm:intrinsicfreeness} and~\ref{thm:universalityspectrum}. Our goal here is to convey how Theorem~\ref{thm:iteratedfreenck} is easy to use, and to illustrate the notion of  \emph{chaos of combinatorial type}.} The key random matrix estimate in~\cite{Hopkins2015TensorPC,Hopkins2018StatisticalIA} is to bound $\EE \| \sum_{i=1}^n W_i \otimes W_i\|$, where $W_1,\dots,W_n$ are i.i.d. $d\times d$ standard Wigner matrices. After decoupling and treating the square terms separately (see~\cite{bandeira2025matrixchaos}) the resulting matrix chaos is given by
        \begin{equation*}
            Y = \sum_{i \in [n],\, j_1, k_1, j_2, k_2 \in [d]} \1_{(j_1, k_1) \neq (j_2, k_2)}\ g_{i,j_1,k_1}^{(1)} g_{i,j_2,k_2}^{(2)}\, 
            e_{j_1}\otimes e_{j_2} \otimes e_{k_1}^\top\otimes e_{k_2}^\top.
        \end{equation*}
        To illustrate the notion of a chaos of combinatorial type, let us compute the norm of one of the flattenings: 
        \begin{eqnarray*}
        \left\|
        \flatta{\{1,2,3\}}{\{4\}}
        \right\| & =& 
        \Bigg\|
        \sum_{i \in [n],\, j_1, k_1, j_2, k_2 \in [d]}
        \1_{(j_1, k_1) \neq (j_2, k_2)}\ e_i\otimes e_{j_1} \otimes e_{k_1} \otimes e_i\otimes e_{j_2} \otimes e_{k_2} \otimes \Big( e_{j_1}\otimes e_{j_2} \otimes e_{k_1}^\top\otimes e_{k_2}^\top\Big)
        \Bigg\| \\
        & \leq  & \Bigg\|
        \sum_{i \in [n],\, j_1, k_1, j_2, k_2 \in [d]}
        e_i\otimes e_{j_1} \otimes e_{k_1} \otimes e_i\otimes e_{j_2} \otimes e_{k_2} \otimes \Big( e_{j_1}\otimes e_{j_2} \otimes e_{k_1}^\top\otimes e_{k_2}^\top\Big)
        \Bigg\|\\
        & =  & \Bigg\|
        \sum_{i \in [n],\, j_1, j_2, \in [d]}
        e_i\otimes e_{j_1} \otimes e_{j_2}
        \Bigg\| = \sqrt{nd^2},
        \end{eqnarray*}
        where the key simplifications is that $e_i\otimes e_i$ can be replace by $e_i$ and $\sum_i e_i\otimes e_i^\top = I$. This gives a straightforward algorithm to compute the norm of flattenings of chaoses of combinatorial type where it suffices to count indices and whether they appear as row or column index (we refer to~\cite{bandeira2025matrixchaos} for a detailed description of this algorithm and an actual definition of \emph{combinatorial type}). After using this procedure to compute $\sigma(\AAA)$ and $v(\AAA)$ (see~\cite{bandeira2025matrixchaos}), Theorem~\ref{thm:iteratedfreenck} yields
        \[
\EE \| Y\| \lesssim d\sqrt{n} + \log(d)^2(d \maxV \sqrt{n}),
        \]
which implies that the SOS degree 6 in~\cite{Hopkins2015TensorPC,Hopkins2018StatisticalIA} succeeds at $\lambda\sim\Omega(n^{-3/4})$, without logarithmic factors.

\begin{remark}
Another important line of work is to provide lower bounds under the SOS framework. Showing that no constant-degree SOS is able to solve a problem in high dimensional statistics is considered to be very strong evidence for the computational hardness of the problem. The current leading approach to provide such lower bounds is pseudo-calibration, which involves analyzing the spectrum of a matrix chaos, usually decomposed in so-called Graph Matrices~\cite{Meka2015SumofsquaresLB,ahn2016graph,Potechin2020MachineryFP}. While the techniques in~\cite{bandeira2025matrixchaos} can be used to bound the spectral norm of graph matrices, it is not at the moment clear whether they can be used to bypass the decomposition. It appears that this would require one to understand the spectral distribution of couples matrix chaoses, not just the spectral norm. We leave this for future endeavors.
\end{remark}

\subsection{Matrix Spencer Conjecture:}\label{sec:matrixspencer}

Another notable application of Theorem~\ref{thm:intrinsicfreeness} is in the remarkable progress of Bansal, Jiang, and Meka~\cite{bansal2023matrixspencer} on the matrix Spencer conjecture.

\begin{conjecture}[Matrix Spencer~\cite{Zouzias2011AMH,Meka-WindowsOnTheory-MSpencer,Afonso_10L42P,RandomstrasseProblems2024}]\label{conj:matrixspencer}
    There exists a positive universal constant $C$ such that, for all positive integers $n$, and all choices of $n$ self-adjoint $n\times n$ real matrices $A_1,\dots,A_n$ satisfying, for all $i\in[n]$, $\|A_i\|\leq 1$ 
    the following holds
    \begin{equation}\label{eq:matrixspencerbound}
    \min_{\varepsilon \in \{\pm1\}} \left\| \sum_{i=1}^n\varepsilon_iA_i\right\|\leq C\sqrt{n}.
    \end{equation}
\end{conjecture}

We note that, by~\eqref{eq:nck1}, $\min_{\varepsilon \in \{\pm1\}} \left\| \sum_{i=1}^n\varepsilon_iA_i\right\|\lesssim \EE \left\| \sum_{i=1}^ng_iA_i\right\|\lesssim \sqrt{\log d} \left\| \sum_{i=1}^nA_i^2\right\|^{\frac12} \lesssim \sqrt{\log n} \sqrt{n}$. Furthermore, if the matrices $A_i$ commute the conjecture reduces to Spencer's seminal ``six standard deviations suffice'' theorem~\cite{Spencer1985Six} (but taking a random choice of signs is not enough). On the other extreme, we anticipate that if the matrices $A_1,\dots,A_n$ behave sufficiently ``freely'' then $\EE \left\| \sum_{i=1}^ng_iA_i\right\|\lesssim \sqrt{n}$.\footnote{The problem is already interesting in the particular case when $A_1,\dots,A_n$ correspond to the regular representation of a group on $n$ elements. It is known to hold for simple groups~\cite{BANDEIRA-CaleyGaussian2024} (by using Spencer's classical argument in commutative groups and a random choice of signs in non-commutative ones) but not known to hold for general finite groups.}  While Conjecture~\ref{conj:matrixspencer} remains open,~\cite{bansal2023matrixspencer} provided a proof in the case where the rank of each of the matrices $A_k$ is at most $\frac{n}{(\log n)^3}$. At a high-level, Bansal et al~\cite{bansal2023matrixspencer} consider a projection of the random matrix $\sum_{i=1}^ng_iA_i$ to a particular subspace where Theorem~\ref{thm:intrinsicfreeness} can be used (showing that $\sum_{i=1}^ng_iA_i$ behaves freely in that subspace) while being high-dimensional enough to still guarantee the validity of~\eqref{eq:matrixspencerbound}.

\section*{Acknowledgments.}

I would like to express my gratitude to all my collaborators, in particular the ones with whom I have collaborated in this line of work: March Boedihardjo, Ramon van Handel, Giorgio Cipolloni, Dominik Schr\"{o}der, Kevin Lucca, and Petar Nizi\'{c}-Nikolac. A special thanks to Ramon, with whom I have been thinking about random matrices for over a decade.
I would also like to thank Joel Tropp for posing the problem of understanding the dimensionality dependence in matrix concentration in a workshop in Oberwolfach in 2014~\cite{GrossKrahmerWardWinter2014OWR}. I was a graduate student in the workshop at the time (and a user of these inequalities), and that question sparked my interest in this topic.
Last but not least, I would like to thank 
Chiara Meroni
and Almut R\"{o}dder who, together with my collaborators mentioned above, made many comments and suggestions that greatly improved this manuscript.

\bibliographystyle{alpha}
\bibliography{icm-afonso}

\end{document}